\documentclass{scrartcl}
\usepackage{mathtools}
\usepackage[backend=biber,style=alphabetic,maxbibnames=9]{biblatex}
\usepackage{amssymb}
\usepackage{bbm}
\usepackage{tikz-cd}
\usepackage{amsthm}
\usepackage{comment}
\usepackage{hyperref}
\usepackage{cleveref}
\usepackage{todonotes}
\usepackage{csquotes}

\bibliography{reference}

\newtheorem*{theoremN}{Theorem}
\newtheorem{theorem}{Theorem}[section]
\newtheorem{proposition}[theorem]{Proposition}
\newtheorem{lemma}[theorem]{Lemma}
\newtheorem{thmx}{Theorem}

\theoremstyle{definition}
\newtheorem{remark}[theorem]{Remark}
\newtheorem{definition}[theorem]{Definition}

\newtheorem{construction}[theorem]{Construction}
\newtheorem{notation}[theorem]{Notation}
\newtheorem*{definitionI}{Definition}

\title{A genuine equivariant recognition principle for finite groups}
\author{Branko Juran \thanks{Department of Mathematical Sciences, University of Copenhagen, Denmark \\ bj@math.ku.dk}}
\date{}

\DeclareMathOperator{\Ar}{Ar}
\newcommand{\Alg}[2]{\mathrm{Alg}_{#1}\left( #2 \right)}
\newcommand{\Alggrp}[2]{\mathrm{Alg}_{#1}^{\mathrm{grp}}\left( #2 \right)}
\DeclareMathOperator{\BarC}{Bar}

\newcommand{\ParBigCat}{\underline{\mathrm{CAT}}}
\newcommand{\ParCat}{\underline{\mathrm{Cat}}}

\newcommand{\CC}{\underline{\mathcal C}}
\newcommand{\DD}{\underline{\mathcal{D}}}
\newcommand{\EE}{\underline{\mathcal E}}

\newcommand{\colors}{\underline{\mathrm{col}}}
\DeclareMathOperator*{\colim}{colim}
\DeclareMathOperator{\const}{const}

\newcommand{\Disk}[1]{\underline{\mathrm{Disk}}_{#1}}
\newcommand{\Diskiso}[1]{\underline{\mathrm{Disk}}_{#1}^{\mathrm{i}}}

\newcommand{\Diskisoext}[1]{\underline{\mathrm{Disk}}_{#1}^{\mathrm{i,e}}}

\newcommand{\Env}{\underline{\mathrm{Env}}}
\newcommand{\E}[1]{\mathbb{E}_{#1}}

\newcommand{\Eisoext}[1]{\E{V}^{\mathrm{i,e}}}
\newcommand{\Eiso}[1]{\E{#1}^{i}}
\newcommand{\fgt}[2]{\mathrm{fgt}^{#1}_{#2}}
\newcommand{\fgtOne}[1]{\mathrm{fgt}^{#1}}

\newcommand{\Fin}[1]{\underline{\mathrm{Fin}}^{#1}}
\newcommand{\FinN}{\underline{\mathrm{Fin}}}
\newcommand{\Fun}{\mathrm{Fun}}
\newcommand{\GFun}{\underline{\Fun}}
\newcommand{\Free}[1]{\mathrm{Free}^{#1}}
\newcommand{\FreeTwo}[2]{\mathrm{Free}^{#1}_{#2}}
\newcommand{\GSpaces}{\underline{\Spaces}}
\newcommand{\BGSpaces}{\underline{\Spaces}_\ast}
\newcommand{\grp}{\mathrm{grp}}
\newcommand{\grpcore}[1]{#1^{\mathrm{core}}}
\newcommand{\GrpCompl}[1]{\mathrm{GrpCompl}^{#1}}
\DeclareMathOperator{\map}{map}
\newcommand{\id}{\mathrm{id}}
\DeclareMathOperator{\lax}{lax}
\newcommand{\Man}[1]{\underline{\mathrm{Man}}_{#1}}

\newcommand{\ManUnfr}{\underline{\mathrm{Man}}}
\newcommand\noloc{%
  \nobreak
  \mspace{6mu plus 1mu}
  {:}
  \nonscript\mkern-\thinmuskip
  \mathpunct{}
  \mspace{2mu}
}

\newcommand{\Nm}[2]{\mathrm{Nm}_{#2}^{#1}}
\newcommand{\OO}{\underline{\mathcal O}}
\newcommand{\PP}{\underline{\mathcal P}}
\newcommand{\Op}[1]{\mathrm{Op}^{#1}}

\newcommand{\overslice}[2]{{#1}_{{/#2}}}
\newcommand{\op}{\mathrm{op}}
\newcommand{\R}{\mathbb R}
\newcommand{\res}[2]{\mathrm{res}^{#1}_{#2}}
\newcommand{\Set}{\underline{\mathrm{Set}}}
\newcommand{\Spaces}{\mathcal S}
\newcommand{\SymMonCat}[1]{\mathrm{Cat}^{#1-\otimes}}
\DeclareMathOperator{\triv}{triv}

\let\originalleft\left
\let\originalright\right
\renewcommand{\left}{\mathopen{}\mathclose\bgroup\originalleft}
\renewcommand{\right}{\aftergroup\egroup\originalright}

\DeclareFieldFormat[article,incollection,inproceedings]{title}{\mkbibemph{#1}}

\begin{document}

\maketitle
\vspace{-25pt}
\begin{abstract}
    For $G$ a finite group and $V$ a finite dimensional real $G$-representation, there is a $G$-operad $\E{V}$ defined using embeddings of $V$-framed $G$-disks such that
    for any based $G$-space $X$, there is a naturally defined $\E{V}$-algebra structure on the $V$-fold space $\Omega^V X$.
    
    Given an $\E{V}$-algebra in $G$-spaces and a subgroup $H$ of $G$, the fixed points $A^H$ carry the structure of an $\E{\dim V^H}$-algebra in spaces. 
    We prove that an $\E{V}$-algebra is equivalent to a $V$-fold loop space if and only if $A^H$ is group-like for all $H$ such that $\dim V^H \ge 1$. This generalizes a result by Guillou and May by removing the assumption that $V$ contains a trivial summand. 
    They observed that the equivariant recognition principle follows from an equivariant version of the approximation theorem, stating that $\Omega^V \Sigma^V X$ is the free group-like $\E{V}$-algebra on a based $G$-space $X$. This has been proven by Hauschild in the case that $V$ contains a trivial summand and by Rourke and Sanderson in the case that $X$ is $G$-connected. Our proof proceeds by showing that the equivariant approximation theorem holds for all $G$-representations $V$ and all based $G$-spaces $X$. 
    %The proof becomes more complicated if $V$ does not contain a trivial summand because the fixed points of an $\E{V}$-algebra $A$ do not admit a natural $\E{1}$-structure and the group completion functor is therefore not given by $\Omega B(-)$ on all fixed points.
    %In our proof, we instead make use of the fact that the fixed points $A^G$ are acted on by the equivariant factorization homology $\int_{V \setminus \{0\}} A$ of $A$ over the punctured representation. We then use  equivariant nonabelian Poincaré duality due to Horev, Klang and Zou to compute this equivariant factorization homology in the case of the free group-like $\E{V}$-algebra.
\end{abstract}
\tableofcontents

\section{Introduction}
When May \cite{may72} introduced the notion of an operad, one of the main motivations was to encode the multiplicative structure on the $n$-fold loop space $\Omega^n X$ of a based space~$X$ which gives rise to the multiplicative structure on the homotopy group $\pi_n(X)$ of that space. This $n$-fold loop space $\Omega^n X$ of a space $X$ admits a homotopy-coherent multiplicative structure, making it an \emph{$\E{n}$-algebra in spaces}. Those $\E{n}$-algebras later started appearing in other context, including higher algebra and the study of configuration spaces of manifolds.

One of the very first and fundamental result in the theory of those operads is May's recognition principle, classifying which of the $\E{n}$-algebras come from $n$-fold loop spaces:
\begin{theoremN}[May] An $\E{n}$-algebra in spaces $A$ is equivalent to an $n$-fold loop space if and only if $\pi_0(A)$ is a group.
\end{theoremN}
This was proven by May \cite{may72} in the case that $\pi_0(A)$ is trivial (and in a different framework by Boardman and Vogt \cite{bv73}) and Segal \cite{seg73} provided the missing input, the approximation theorem, to deduce the general case, as explained in {\cite[pp. 487, footnote~21]{clm76}}.
The hard part is the \enquote{only if}-direction, showing that it suffices that $\pi_0(A)$ is a group in order to construct a space $B^n A$ such that $A \cong \Omega^n B^n A$.
The more detailed version of the above theorem does that explicitly by constructing the delooping $B^n A$ of an arbitrary $\E{n}$-algebra $A$ and then shows that there is a natural map $A \to \Omega^n B^n A$ which is an equivalence if and only if $\pi_0(A)$ if a group. In general this map is the so-called \emph{group completion}, the initial map into an $\E{n}$-algebra for which $\pi_0$ is a group.

The goal of this paper is to prove a genuine equivariant version of the above theorem, generalizing previous conditional results by May and Guillou. 
We will study \emph{genuine $G$-spaces}, objects represented by topological spaces with an action of a finite group $G$. The role of the $\E{n}$-operad is taken over by the $G$-operads $\E{V}$ where~$V$ can be any $n$-dimensional real $G$-representation. 
Apart from applications to equivariant loop space theory, these $G$-operads have been used to study equivariant factorization homology \cite{hor19} \cite{zou23} \cite{hkz24}, 
equivariant and real versions of topological Hochschild homology \cite{hor19} \cite{dmpr21} and equivariant versions of the Hopkins-Mahowald theorem \cite{hw20} \cite{lev22}.

Given a $G$-representation $V$ and a based $G$-space $X$, we can consider the $V$-fold loop space $\Omega^V X$, that is based maps from the one-point-compactification $S^V$ of $V$ into $X$. 
This $V$-fold loop space carries the structure of an $\E{V}$-algebra. 
A genuine equivariant recognition principle must find a list of necessary conditions on an $\E{V}$-algebra to be equivalent to a $V$-fold loop space. 
For a subgroup $H$ of $G$, the fixed points $A^H$ of an $\E{V}$-algebra $A$ carry the structure of an $\E{\dim V^H}$-algebra in spaces. In particular, we only obtain a monoid structure on $\pi_0(A^H)$ if $\dim V^H \ge 1$. 
in the case that $A=\Omega^V X$ is a $V$-fold loop space, this monoid is a group.
The goal of this paper is to show that this necessary condition on $A$ to be equivalent to a $V$-fold loop space is also sufficient. 

Before we can state the main theorem, we need the following two definitions:
\begin{definitionI}
    Let $A$ be an $\E{V}$-algebra in $G$-spaces. We say that $A$ is \emph{group-like} if $\pi_0(A^H)$ is a group for all $H$ such that $\dim V^H \ge 1$.
    We denote the full subcategory of group-like $\E{V}$-algebras by $\Alggrp{\E{V}}{\GSpaces} \subset \Alg{\E{V}}{\GSpaces}$.

    Given an $\E{V}$-algebra $A$, we say that a map $A \to A^{\grp}$ to another $\E{V}$-algebra $A^{\grp}$ exhibits $A^{\grp}$ as the group completion of $A$ if it is an initial map to a group-like $\E{V}$-algebra. 
\end{definitionI}
\begin{definitionI}
    For $V$ a real $G$-representation, a based $G$-space $X$ is called $V$-connective if its $H$-fixed points $X^H$ are $(\dim V^H -1)$-connected for every subgroup $H$ of $G$.
\end{definitionI}
We will now state the genuine equivariant recognition principle in the form which also gives an explicit description of the group completion functor for non-group-like $\E{V}$-algebras.
\begin{thmx}[Recognition principle] \label{thm:recognition-principle}
    There is an adjunction 
    \[ \Omega^V \colon \Spaces^G_\ast \leftrightarrows \Alg{\E{V}}{\GSpaces} \noloc B^V  \] between the category of based $G$-spaces and the category of $\E{V}$-algebras in $G$-spaces such that
    \begin{itemize}
        \item For $A$ an $\E{V}$-algebra in $G$-spaces, the unit of the adjunction 
        \[ A \longrightarrow \Omega^V B^V A \]
        is an equivalence if and only if $A$ is group-like. In general, it exhibits the target as the group completion of the source. 
        \item For $X$ a based $G$-spaces, the counit of the adjunction 
        \[ B^V \Omega^V X \longrightarrow X \] 
        is an equivalence if and only if $X$ is $V$-connective. In general, it exhibits the source as the $V$-connective cover of the target. 
    \end{itemize}
    In particular, the above adjunction restricts to an equivalence of categories \[ \left( \Spaces^G_\ast \right)_{\ge V} \cong \Alggrp{\E{V}}{\GSpaces} \] between $V$-connective based $G$-spaces and group-like $\E{V}$-algebras. 
\end{thmx}
It is an easy consequence of this theorem that the free group-like $\E{V}$-algebra on a based $G$-space $X$ must be given by the $V$-fold loop space $\Omega^V \Sigma^V X$.
Following May's strategy of using the monadicity theorem, it is actually possible to deduce the recognition principle from this special case. We therefore first prove an equivariant version of the so-called approximation theorem:
\begin{thmx}[Approximation theorem] \label{thm:approximation-theorem}
    For $X$ a based $G$-space, the natural map
    \[ \Free{\E{V}} X \longrightarrow \Omega^V \Sigma^V X \] from the free $\E{V}$-algebra on $X$ to the $V$-fold loop space of the $V$-fold suspension of $X$ exhibits the target as the group completion of the source. 
\end{thmx}
The non-equivariant version of the approximation theorem is the aforementioned result due to Segal \cite{seg73}. 
The equivariant version is known in the following two cases: 
Hauschild \cite{hau80} proved it in the case that $V$ contains a trivial summand, in which the group completion can be computed by applying $\Omega B(-)$ on all fixed points. 
In the cited paper, he only provides details for the case $X=S^0$ but remarks that the general case works similarly.
Rourke and Sanderson \cite{rs00} gave another proof of the same special case and also proved the result for arbitrary $V$, provided that $X$ is $G$-connected. In the latter case, the left hand side already is group-like and the map appearing in \Cref{thm:approximation-theorem} is an equivalence. 
This was used by Guillou and May \cite{gm17} to deduce \Cref{thm:recognition-principle} in the same cases, i.e. if $V$ contains a trivial summand or that the unit $A \to \Omega^V B^V A$ is an equivalence if the $\E{V}$-algebra $A$ is connected.
A version of the recognition principle for $G=C_2$ and $V$ being the sign representation appeared in Stiennon's thesis \cite{sti13} and later in work by Moi \cite{moi20}. Both of them prove that a group-like simplicial monoid with anti-involution models an equivariant loop space. 
%Also, a potential generalization to arbitrary tangential structures in the spirit of \cite{sw03} must also specialize to (and potentially start from) a recognition principle for all representations.

\subsection*{Proof strategy}
We will now elaborate on the proof strategy of the equivariant approximation theorem (\Cref{thm:approximation-theorem}). 
The recognition principle (\Cref{thm:recognition-principle}) is a rather formal consequence of this result using the monadicity theorem.

As mentioned earlier, the main difficulty is that for the subgroups $H$ of $G$ for which $V^H=\{0\}$, the fixed points $A^H$ of an $\E{V}$-algebra $A$ do not generally admit the structure of an $\E{1}$-algebra. 
In contrast to the case where $V$ contains a trivial summand this means that the $\E{V}$-group completion functor is not simply given by group completion on all fixed points. 
However, we use that those $H$-fixed points are acted on by the fixed points of the equivariant factorization homology $\left(\int_{V \setminus \{0\}} A \right)^H$. 

In order to study the behavior of the group completion on those two types of fixed points separately, we define a $G$-operad $\Eisoext{V}$ which lies in between $\E{0}$ and $\E{V}$. For subgroups $H$ such that $\dim V^H \ge 1$, the collection of $H$-fixed points of an $\Eisoext{V}$-algebra have the same structure as for an $\E{V}$-algebra. However, for the $H$-fixed points for $H$ such that $\dim V^H = 0$, they really only admit the structure of an $\E{0}$-algebra in spaces without an action of the equivariant factorization homology. 

Using this, we can split up the computation of the free group-like $\E{V}$-algebra into three steps:
We first compute the free $\Eisoext{V}$-algebra on $X$, then its group completion as an $\Eisoext{V}$-algebra and finally the free $\E{V}$-algebra on that group completed $\Eisoext{V}$-algebra.

For the first step, we use the explicit formula for a free algebra for an operad to see that the $H$-fixed points with $\dim V^H \ge 1$ of the free $\Eisoext{V}$-algebra on a based $G$-space $X$ are given by a certain equivariant configuration space with labels in $X$. Moreover, the $H$-fixed points for $H$ with $\dim V^H = 0$ are just given by $X^H$.

In the second step, we use that the condition $\dim V^H \ge 1$ is equivalent to $H$-representation $\res{G}{H} V$ containing a trivial summand. We might therefore use the approximation theorem of Hauschild \cite{hau80} to deduce that the $H$-fixed points of the $\Eisoext{V}$-group completion of the free $\Eisoext{V}$-algebra on $X$ are equivalent to $\left( \Omega^V \Sigma^V X \right)^H$. 
Moreover, the group completion does not change the other $H$-fixed points at all.

In the third and final step, we have to compute the free $\E{V}$-algebra on that group completed $\Eisoext{V}$-algebra. 
This time, this will not change any $H$-fixed points for which $\dim V^H \ge 1$.
Using that this holds for all $H$ appearing as isotropy groups of the $G$-manifold $V \setminus \{0\}$ and our previous computation, we will see that the $H$-fixed points for $\dim V^H = 0$ are equipped with a free action of the equivariant factorization homology $\left(\int_{V \setminus \{0\}} \Omega^V \Sigma^V X\right)^H$.
Finally, we will use equivariant nonabelian Poincaré duality, due to Horev, Klang and Zou, \cite{hkz24} to compute this factorization homology. %finding that it is indeed equivalent to $(\Omega^V \Sigma^V X)^H$ as desired.

\subsection*{Acknowledgments}
I would like to thank my PhD supervisor Jesper Grodal for his help and guidance while writing this paper. 
I am grateful to Gabriel Angelini-Knoll, Robert Burklund, Jan Steinebrunner and Nat(h)alies Stewart and Wahl for helpful discussions. I would like to especially thank Natalie Stewart for comments on a previous draft of this article. The author was supported by the DNRF
through the Copenhagen Center for Geometry and Topology (DNRF151). 

\section{Preliminaries}
    Throughout this paper, we fix a finite group $G$ and an $n$-dimensional real $G$-representa-tion~$V$.
\subsection{\texorpdfstring{$G$}{G}-symmetric monoidal \texorpdfstring{$G$}{G}-categories and \texorpdfstring{$G$}{G}-operads}
    We use the theory of $\infty$-categories as developed by Lurie in \cite{lur09} and \cite{lur17}. We use the term \emph{category} to refer to an $\infty$-category. 
    The category of spaces (or animae, homotopy types, ...) is denoted by $\Spaces$.
    Moreover, we use the theory of parameterized homotopy theory developed by Barwick, Glasman, Dotto, Nardin and Shah \cite{bdgns24} and in particular parameterized operads as developed by Nardin and Shah in \cite{sha23} and \cite{ns22}. 
    In particular, we assume the reader to be familiar with $G$-symmetric monoidal $G$-categories and $G$-operads. 

    We will mostly omit the $G$ from the notation and underline categories to indicate that they are parameterized, e.g. the $G$-category of (genuine) $G$-spaces is denoted by $\GSpaces$, its based version by $\BGSpaces$. Their fixed point categories $\GSpaces^G=\Spaces^G$ and $\BGSpaces^G=\Spaces^G_\ast$ then recover the categories of (genuine) $G$-spaces and based (genuine) $G$-spaces, respectively.

    For $\CC$ a $G$-symmetric monoidal $G$-category and $K \subset H$ nested subgroups of $G$, we use the notation $\Nm{H}{K} \colon \CC^K \longrightarrow \CC^H$ for its norm functor.

    We denote the category of $G$-operads by $\Op{G}$ and the category of $G$-symmetric monoidal $G$-categories and $G$-symmetric monoidal functors by $\SymMonCat{G}$. 

    The inclusion $\SymMonCat{G} \to \Op{G}$ admits a left adjoint, the $G$-envelope functor
    \[ \Env \colon \Op{G} \longrightarrow \SymMonCat{G} \,. \]
    Using that the $G$-envelope of the terminal $G$-operad is given by the $G$-symmetric monoidal $G$-category of finite $G$-sets $\FinN$ with its cocartesian $G$-symmetric monoidal structure, this functor factors through the slice category $\overslice{\SymMonCat{G}}{\FinN}$. 
    We will make use of the following theorem:
    \begin{theorem}[Barkan, Haugseng, Steinebrunner {\cite[Cor. 5.2.15, Lem. 5.2.16]{bhs24}}] \label{thm:operad=symmon}
        The $G$-envelope functor induces a fully faithful functor \[ \Op{G} \longrightarrow \overslice{\operatorname{Cat}^{G-\otimes}}{\FinN} \] from the category of $G$-operads to the slice category of $G$-symmetric monoidal $G$-categories and $G$-symmetric monoidal functors over the $G$-symmetric monoidal $G$-category of finite $G$-sets $\FinN$. 
        The essential image is spanned by the $G$-symmetric monoidal functors $F \colon \CC \to \FinN$ such that the squares
        \[ \begin{tikzcd}[sep=huge]
            \CC^H \times \CC^H \ar[r,"{\otimes}"] \ar[d,"{F^H \times F^H}"] & \CC^H \ar[d,"{F^H}"] \\
            \Fin{H} \times \Fin{H} \ar[r,"{\sqcup}"] & \Fin{H}
        \end{tikzcd} \quad \text{and} \quad \begin{tikzcd}[sep=huge]
            \CC^K \ar[r,"{\Nm{H}{K}}"] \ar[d,"F^K"] & \CC^H \ar[d,"{F^H}"] \\
            \Fin{K} \ar[r,"{H \times_K (-)}"] & \Fin{H}
        \end{tikzcd} \]
        are pullbacks for all nested subgroups $K \subset H$ of $G$.
    \end{theorem}
    \begin{remark} \label{prop:criterion-sub-operad}
        The above theorem in particular implies the following: Given a $G$-operad $\OO$, for a (not necessarily full) $G$-symmetric monoidal $G$-subcategory $\CC$ of a $\Env(\OO)$, the functor $\CC \subset \Env(\OO) \to \Fin{G}$ is contained in the essential image of the $G$-envelope functor from \Cref{thm:operad=symmon} if the following holds true for all nested subgroups $K \subset H$:
        \begin{itemize}
            %\item For all $c$ and $d$ in $\Env(\OO)^H$: If $c \otimes d$ is contained in $\CC^H$, then so are $c$ and $d$.
            \item For all $f \colon c \to c^\prime$ and $g \colon d \to d^\prime$ in $\Env(\OO)^H$: If $f \otimes g \colon c \otimes d \to c^\prime \otimes d^\prime$ is contained in $\CC^H$, then so are $f$ and $g$ (and in particular their source and target).
            %\item For all $c$ in $\Env(\OO)^K$: If $\Nm{H}{K} (c)$ is in $\CC^H$, then $c \in \Env(\OO)^H$.
            \item For all $f \colon c \to c^\prime$ in $\Env(\OO)^K$: If $\Nm{H}{K} (f) \colon \Nm{H}{K} (c) \to \Nm{H}{K} (c^\prime)$ is in $\CC^H$, then $f$ is in $\CC^K$ (and in particular its source and target)
        \end{itemize}
    \end{remark}
    Using the above theorem and in particular the remark, we will not work with $G$-operads directly but rather replace them with their $G$-envelope.

    Given a $G$-symmetric monoidal $G$-category $\CC$ and a $G$-operad $\OO$, we write \[ \Alg{\OO}{\CC} = \Fun_{\Op{G}}(\OO,\CC) \cong \Fun^{G-\otimes}(\Env(\OO),\CC) \] for the category of $\OO$-algebras in $\CC$. 
    This category is the fixed point category of a $G$-category of algebras.

The following result gives an explicit formula for computing free algebras: 
\begin{proposition}
\label{prop:formula-Kan-extension}
    Let $f \colon \OO \to \PP$ be a map of $G$-operads. Then precomposition 
    \[ f^\ast \colon \Alg{\PP}{\GSpaces} \longrightarrow \Alg{\OO}{\GSpaces} \] admits a left adjoint, which can be computed as a left Kan extension of the induced maps out of the $G$-envelope, i.e. the Beck Chevalley transform of the following diagram is invertible, making the diagram commute:
    \[ \begin{tikzcd}[column sep = huge]
        \Alg{\OO}{\GSpaces} \ar[r,"f_!"] \ar[d] & \Alg{\PP}{\GSpaces} \ar[d] \\
        \Fun(\Env(\OO),\GSpaces) \ar[r,"{\Env(f)_!}"] & \Fun(\Env(\PP),\GSpaces)
    \end{tikzcd}  \]
    where the upper arrow is the so-called \emph{operadic left Kan extension} along $f$, the bottom arrow is left Kan extension along the underlying functor of $\Env(f)$ and the vertical arrows are forgetful functors.
\end{proposition}
\begin{proof}
    This is a special case of \cite[Lem. 3.45]{llp25}, using that the $G$-category of $G$-spaces is distributive.% by \cite[Rem. 3.15]{llp25} and \cite[Prop. 3.2.5]{ns22}.
\end{proof}
%\begin {remark} \label{rem:kan-extension-other-target}
%    Using the parameterized tensor product of presentable categories from \cite{nar17}, the above results generalize from functors into $\GSpaces$ to functors into any presentably $G$-symmetric monoidal $G$-category. 
%\end{remark}
Given a $G$-operad $\OO$, write $\colors(\OO)$ for its $G$-category of colors. Finally, we recall the following proposition. %We will need the notion of a distributive $G$-symmetric monoidal category from \cite[Def. 3.2.4]{ns22}. The only important thing for us is that the category of $G$-spaces is distributive by \cite[Prop. 3.2.5]{ns22}. 
We write \[ \fgtOne{\OO} \colon \Alg{\OO}{\GSpaces} \longrightarrow \Fun(\colors(\OO),\GSpaces) \] for the forgetful functor.
\begin{proposition}[{\cite[Thm. 5.1.4(2), Cor. 5.1.5]{ns22}}] \label{prop:monadicity-forgetful} \label{that is not really the precise statement is it. apparenty is, check nathalie}
    For $\OO$ a $G$-operad, the forgetful functor 
    \[ \fgtOne{\OO} \colon \Alg{\OO}{\GSpaces} \longrightarrow \Fun(\colors(\OO),\GSpaces) \] is a conservative right adjoint preserving geometric realizations. In particular, it is monadic. 
\end{proposition}
\subsection{The equivariant little disk operad  \texorpdfstring{$\E{V}$}{EV}, \texorpdfstring{$V$}{V}-fold loop spaces and equivariant nonabelian Poincaré duality}
    We will now turn towards more geometric constructions. We will recall the construction of the equivariant little disk operad $\E{V}$ and how to realize loop spaces as algebras over those $G$-operads. We will then recall the statement of nonabelian Poincaré duality. 

    The $G$-symmetric monoidal $G$-category of $n$-dimensional $G$-manifolds from \cite{hor19} is denoted by $\ManUnfr$. Its $V$-framed version is denoted by $\Man{V}$. The $G$-symmetric monoidal subcategory of finite disjoint unions of $V$-framed $G$-disks will be denoted by $\Disk{V}$. 
    The $G$-symmetric monoidal functor $\Disk{V} \to \FinN$ sending a disk to its $G$-set of equivariant path components is contained in the essential image of the sliced $G$-envelope functor from \Cref{thm:operad=symmon}, so that $\Disk{V}$ is the $G$-envelope of a $G$-operad $\E{V}$, see \cite[Prop. 3.7.4, Prop. 3.9.8]{hor19}. Given another $G$-symmetric monoidal $G$-category $\CC$, we might therefore write 
    \[ \Alg{\E{V}}{\CC} = \Fun^{G-\otimes}(\Disk{V},\CC) \] for the category of $\E{V}$-algebras in $\CC$.
Let us start with some general observations about these little disk operads, starting with a non-equivariant description of $\E{V}$-algebras for $V$ being the trivial representation.
For $n \ge 0$, let $\triv^n$ denote the $n$-dimensional trivial representation. Let $\mathrm{Disk}_n$ denote the non-equivariant category of $n$-dimensional disks. Then there is a functor $\mathrm{Disk}_n \to \left(\Disk{\triv^n}\right)^G$ into the category of $G$-disks framed in $\triv^n$, equipping a disk with trivial $G$-action. 
\begin{lemma}[{\cite[Lem. 7.2.1]{hor19}}]
    Let $\CC$ be a $G$-symmetric monoidal $G$-category. The functor  
    \begin{align*} 
        & \Alg{\E{\triv^n}}{\GSpaces} \\
        \cong{}& \Fun^{G-\otimes}\left(\Disk{\triv^n},\CC\right) \\
         \xrightarrow{(-)^G}{}& \Fun^{\otimes}\left(\left(\Disk{\triv^n}\right)^G,\CC^G\right)  \\
        \longrightarrow{}& \Fun^{\otimes}\left( \Disk{n} , \CC^G \right) \cong \Alg{\E{n}}{\CC^G} \end{align*} 
        is an equivalence.
\end{lemma}
We will implicitly use this proposition to identify based $G$-spaces with $\E{0}$-algebras in $G$-spaces.

More generally, there is a functor $\Disk{V^H} \to \Disk{V}^H$ equipping a $\dim V^H$-dimensional non-equivariant disk with the trivial $H$-action and then taking products with the orthogonal complement of $V^H$ in $V$. This induces a functor 
\[ (-)^H \colon \Alg{\E{V}}{\GSpaces} \longrightarrow \Alg{\E{V^H}}{\Spaces} \] where we write $\E{V^H}$ for the non-equivariant operad $\E{\dim V^H}$ to emphasize the functoriality. 

The little disk operads are functorial in representations, so that the inclusion $0 \to V$ induces a forgetful functor 
\[ \Alg{\E{V}}{\GSpaces} \longrightarrow \Alg{\E{0}}{\GSpaces} \cong \Spaces^G_\ast \,. \] This functor can also be described by evaluating the strong symmetric monoidal functor $\Disk{V} \to \GSpaces$ at $V \in \Disk{V}^G$, giving a $G$-spaces which admits a base point coming from the unique map from the empty disk into $V$.
We denote the left adjoint of this functor by 
\[ \Free{\E{V}} \colon \Spaces^G_\ast \longrightarrow \Alg{\E{V}}{\GSpaces} \,, \]
omitting the $\E{0}$ from the notation.

We will now discuss group-like algebras. 
Recall that for $n \ge 1$ and $A$ an $\E{n}$-algebra in spaces, $\pi_0(A)$ naturally admits the structure of a monoid in sets (commutative if $n \ge 2$) and that $A$ is called \emph{group-like} if that monoid is a group.
Let us recall the following definition from the introduction:
\begin{definition}
    Let $A$ be an $\E{V}$-algebra. We say that $A$ is \emph{group-like} if $\pi_0(A^H)$ is a group for all $H$ such that $\dim V^H \ge 1$.
    We denote the full subcategory of group-like $\E{V}$-algebras by $\Alggrp{\E{V}}{\GSpaces} \subset \Alg{\E{V}}{\GSpaces}$.
\end{definition}

Finally, we note that the existence of \emph{some} group completion functor for $\E{V}$-algebras can be proven formally:

\begin{proposition} \label{prop:grpcompl-exists}
    The inclusion \[  \Alggrp{\E{V}}{\GSpaces} \longrightarrow \Alg{\E{V}}{\GSpaces} \] admits a left adjoint. 
\end{proposition}
\begin{proof}
    The subcategory of group-like functors can be written as the class of local objects with respect to the set of morphisms corepresenting the shearing maps on $H$-fixed points, so the claim follows from \cite[Prop. 5.5.4.15]{lur09}.
\end{proof}
\begin{notation}
    We denote the left adjoint to the inclusion described in \Cref{prop:grpcompl-exists} by \[ \GrpCompl{\E{V}} \colon \Alg{\E{V}}{\GSpaces} \longrightarrow \Alggrp{\E{V}}{\GSpaces} \] and call it the \emph{group completion}.
\end{notation}
The problem hence is not to show the existence of this functor, but to explicitely compute it. 
Next, we will explain how to construct an $\E{V}$-algebra structure on a $V$-fold loop space.

\begin{construction} \label{constr:one-point-compactification}
    The following construction is discussed in detail in \cite[Sec. 2.3]{hkz24}, where the reader might find more details.
    There is a $G$-symmetric monoidal functor 
    \[ (-)^+ \colon (\ManUnfr)^{\sqcup} \longrightarrow \left(\left( \BGSpaces \right)^{\op} \right)^{\vee} \] where the target is equipped with the cartesian monoidal structure in the opposite $G$-category of $G$-spaces, i.e. the wedge product/induction. It sends a manifold to the homotopy type of its one point compactification and a map to its induced collapse map. 
\end{construction}
\begin{construction} \label{constr:loop-space}
    We will now recall the definition of the $V$-fold loop space functor from \cite[Def. 6.2.1]{hkz24}. It is defined to be the composite
    \[ \Omega^V \colon \Spaces_\ast^{G} \longrightarrow \Fun^{G-\otimes}\left(\left((\BGSpaces)^{\vee}\right)^{\op},(\GSpaces)^{\times}\right) \longrightarrow \Fun^{G-\otimes}\left(\Disk{V},\left(\GSpaces\right)^{\times}\right) \,. \] Here, the first functor is the Yoneda embedding. As any representable functor preserves $G$-limits, it indeed naturally factors through the category of $G$-symmetric monoidal functor from $(\BGSpaces)^{\op}$ and $\GSpaces$ equipped with their respective cartesian symmetric monoidal structure from \cite[Prop. 5.12]{sha23}. 
    The second functor is the restriction along the $G$-symmetric forgetful functor $\Disk{V} \hookrightarrow \Man{V} \to \ManUnfr$ and the functor $\ManUnfr \to \left((\GSpaces_\ast)^{\vee}\right)^{\op}$ from \Cref{constr:one-point-compactification}.
\end{construction}

Finally, we discuss equivariant factorization homology. 
We restrict our attention to factorization homology with values in the $G$-category of $G$-spaces.%We restrict our attention to $\E{V}$-algebras with values in distributive $G$-symmetric monoidal categories, to ensure that factorization homology is well-behaved with respect to the $G$-symmetric monoidal structures. We will anyway only study equivariant factorization homology with values in $G$-spaces $\GSpaces$.
\begin{definition} 
    %Let $\CC$ be a distributive $G$-symmetric monoidal $G$-category. 
    %Using \Cref{prop:formula-Kan-extension}, we define the equivariant factorization homology as the operadic Kan extension along the inclusion $\Disk{V} \hookrightarrow \Man{V}$
    We define the equivariant factorization homology as the left Kan extension along the inclusion $\Disk{V} \subset \Man{V}$:
    \[ \int \colon \Alg{\E{V}}{\GSpaces} \cong \Fun^{G-\otimes}(\Disk{V},\GSpaces) \longrightarrow \Fun^{G}(\Disk{V},\GSpaces)\longrightarrow \Fun^{G}(\Man{V},\GSpaces) \,, \] 
    %By \cite[Prop. 4.2.2]{hor19}, this operadic Kan extension can be computed as a non-operadic Kan extension, 
    that is for $A$ an $\E{V}$-algebra in $\GSpaces$ and $M$ a $V$-framed manifold, the equivariant factorization homology of $A$ over $M$ is given by:
    \[  \int_M A \cong \colim \left({\overslice{\Disk{V}}{M}} \longrightarrow \Disk{V} \xlongrightarrow{A} \GSpaces  \right) \] 
\end{definition}
\begin{remark}
    Using a lemma similar to \Cref{prop:formula-Kan-extension} (essentially \cite[Thm. 3.39]{llp25}), Horev shows that this Kan extension naturally admits a $G$-symmetric monoidal structure. We will not make use of this enhancement.
\end{remark}
\begin{construction} \label{constr:NAPD}
    Let $X$ be a based $G$-space and let $M$ be a $V$-framed manifold, we will recall from \cite{hkz24} how to construct a natural map
    \[ \int_M \Omega^V X \longrightarrow \map_\ast(M^+,X) \] of $G$-spaces. 
    It follows from \Cref{constr:loop-space} that $\Omega^V X \colon \Disk{V} \longrightarrow \GSpaces$ can actually be extended to a functor out of $\Man{V}$ such that the value at $M$ is $\map_\ast(M^+,X)$. It hence follows that there is a natural transformation
    \[ \left( \overslice{\Disk{V}}{M} \longrightarrow \Disk{V} \xrightarrow{\Omega^VX} \GSpaces \right) \Longrightarrow \const \map_\ast(M^+,X) \,,.  \] 
    The natural map we wanted to construct now arises by using the universal property of parameterized colimits.
\end{construction}
Before we state equivariant nonabelian Poincaré duality, let us recall the following definition from the introduction:
\begin{definition}
    A based $G$-space $X$ is called $V$-connective if its $H$-fixed points $X^H$ are $(\dim V^H -1)$-connected for every subgroup $H$ of $G$. 
    We write $(\Spaces_\ast^G )_{\ge V} \subset \Spaces_\ast^G$ for the full subcategory of $V$-connective spaces. 
\end{definition}
\begin{theorem}[{\cite[Thm. 4.0.1]{hkz24}}, equivariant nonabelian Poincaré duality] \label{thm:NAPD}
    For $X$ a $V$-connective based $G$-space and $M$ a $V$-framed manifold, the natural map from \Cref{constr:NAPD} \[ \int_M \Omega^V X \longrightarrow \map_\ast(M^+,X) \] is an equivalence of $G$-spaces.
\end{theorem}

\section{The approximation theorem}
In this section, we will proof the approximation theorem. We recommend the reader to recall the proof strategy from the introduction.

\subsection{Factoring the free group-like functor}
In this subsection, we will introduce the immediate $G$-operads $\E{0} \subset \Eiso{V} \subset \Eisoext{V} \subset \E{V}$, the latter one we informally described when explaining the strategy of the proof in the introduction. 

Let us introduce a terminology for the subgroups of $H$ of $G$ for which $A^H$ admits a monoid structure for $A$ an $\E{V}$-algebra. 
\begin{definition}
    For $V$ a $G$-representation, we say that a subgroup $H$ of $G$ is \emph{$V$-isotropy} if $\dim V^H \ge 1$. 
\end{definition}
We will now define the operad $\Eisoext{V} \subset \E{V}$. The idea is that in this operad, we do not allow for any operations going from $V$-isotropy subgroups to subgroups which are not $V$-isotropy.
\begin{definition}
    Let $\Diskisoext{V} \subset \Disk{V}$ denote the non-full $G$-symmetric monoidal subcategory spanned by all objects but on $H$-fixed points, but we only take those embeddings $f \colon D_1 \to D_2$ of $\res{G}{H} V$-framed $H$-disks such that for an $H$-path component $D \subset D_2$ such that $D^K \neq \emptyset$ for some non-$V$-isotopy $K<H$, the same is true for all path components in $f^{-1}(D)$.
    
    We can use \Cref{prop:criterion-sub-operad} to see that this is an enveloping algebra of a $G$-operad with a map to $\E{V}$ which we will denote by $\Eisoext{V}$. We therefore denote the category of $\Eisoext{V}$-algebras in a $G$-symmetric monoidal $G$-category $\CC$ by \[ \Alg{\Eisoext{V}}{\CC} = \Fun^{G,\otimes}\left( \Diskisoext{V} , \CC \right) \,. \]
    %The inclusion $\Diskisoext{V} \subset \Disk{V}$ induces a functor 
    %\[ \fgt{\E{V}}{\Eisoext{V}} \colon  \Alg{\E{V}}{\CC} \longrightarrow  \Alg{\Eisoext{V}}{\CC} \,. \]
\end{definition}

It will also be helpful to have the following operad, which does not remember any fixed point data for subgroups which are not $V$-isotropy.
\begin{definition}
    We write $\Diskiso{V}$ for the full $G$-symmetric monoidal subcategory of $\Disk{V}$ (and $\Diskisoext{V}$) spanned on $H$-fixed points by all $H$-disks for which all isotropy groups are $V$-isotropy.
    Using \Cref{prop:criterion-sub-operad}, we see that this again is an enveloping algebra of a $G$-operad with a map to $\Eisoext{V}$ which we will denote by $\Eiso{V}$. 
    For $\CC$ a $G$-symmetric monoidal $G$-category, we hence write \[ \Alg{\Eiso{V}}{\CC} = \Fun^{G,\otimes}\left( \Diskiso{V} , \CC \right) \] for the category of $\Eiso{V}$-algebras in $\CC$. 
\end{definition}
Finally, let us introduce the corresponding notion of group-like algebras for those operads and record that group completion exists. (Even though in this case we will construct those group completions in the proofs of \Cref{prop:grpcompl-le-H} and \Cref{prop:Eisoext-grp-compl-is-levelwise} directly without assuming its a priori existence anyway.)
\begin{definition}
    Following our conventions for the $\E{V}$-operad, given an $\Eisoext{V}$-algebra $A$, we write $A^H$ for the based space obtained by taking $H$-fixed points of the $H$-space which is given by evaluating the $G$-symmetric monoidal functor $\Diskisoext{V} \to \GSpaces$ at the $H$-disk $\res{G}{H} V$, receiving a base point from the unique map from the empty disk into $\res{G}{H} V$. If $H$ is $V$-isotropy, then these fixed points admit a natural $\E{V^H}$-algebra structure and this already is well-defined for an $\Eiso{V}$-algebra.
    We call a $\Eisoext{V}$- (or $\Eiso{V}$-)algebra in spaces $A$ \emph{group-like} if $A^H$ is group-like for all $H$ which are $V$-isotropy and denote the corresponding subcategories of group-like algebras by $\Alggrp{\Eisoext{V}}{\GSpaces}$ and $\Alggrp{\Eiso{V}}{\GSpaces}$, respectively.
    The same argument as in the proof of \Cref{prop:grpcompl-exists} shows that there are left adjoints to the inclusions which we will denote by
    \begin{align*}
        \GrpCompl{\Eisoext{V}} \colon \Alg{\Eisoext{V}}{\GSpaces} & \longrightarrow \Alggrp{\Eisoext{V}}{\GSpaces} 
        \intertext{and}
        \GrpCompl{\Eiso{V}} \colon \Alg{\Eiso{V}}{\GSpaces} & \longrightarrow \Alggrp{\Eiso{V}}{\GSpaces} \,.
    \end{align*}
\end{definition} 
Before we get into the proof, let us record the following fact, which will be useful in order to deal with the $V$-fold loop space functor:
\begin{proposition} \label{prop:monadicity-OmegaV}
    The $V$-fold loop space functor \[ \Omega^V \colon (\Spaces^G_\ast)_{\ge V} \longrightarrow \Spaces^G_\ast \] from $V$-connective based $G$-spaces to based $G$-spaces is conservative and commutes with geometric realizations.
\end{proposition}
\begin{proof}
    The claim about geometric realizations is proven in \cite[Lem. 5.4]{cw91}.

    Let us turn to the conservativity claim. If $V = W \oplus \R$ contains a trivial summand, we might write $\Omega^V  = \Omega^{W} \circ \Omega $ and reduce to the same claim for $W$. We might hence assume that $V$ is fixed point free. 

    Let $f\colon X \to Y$ be a map of $V$-connective $G$-spaces such that $\Omega^V f \colon \Omega^V X \to \Omega^V Y$ is an equivalence. By inducting on the size of the group $G$, we might assume that $f^H \colon X^H \to Y^H$ is an equivalence for all proper subgroups $H$ of $G$. 
    Now consider the cofiber sequence \[ S(V)_+ \longrightarrow S^0 \longrightarrow S^V \,. \] Mapping into $f \colon X \to Y$ yields a map between fiber sequences 
    \[ \begin{tikzcd}
        (\Omega^V X)^G \ar[r]  \ar[d] & X^G \ar[r] \ar[d] & \map(S(V),X)^G \ar[d] \\
        (\Omega^V Y)^G \ar[r]  & Y^G \ar[r] & \map(S(V),Y)^G
     \end{tikzcd} \] 
     The right hand map is an equivalence because $S(V)$ does not have $G$-fixed points, and we already assumed that $f$ is an equivalence on fixed points for all proper subgroups. The left hand map is an equivalence by assumption.
     Finally, $S(V)$ admits an equivariant cell structures with $H$-cells at most of dimension $\dim V^H - 1$. By cell induction, we conclude that any map from $S(V)$ into a $V$-connective space must be trivial. It follows that the base space of the two fibrations is connected and we can conclude that the map $f^G \colon X^G \to Y^G$ is an equivalence, as desired.
\end{proof}
From this we deduce: 
\begin{lemma} \label{lem:monadicity-E-Omega}
    The $V$-fold loop space functor 
    \[ \Omega^V \colon (S_\ast^G)_{\ge V} \longrightarrow \Alg{\E{V}}{\GSpaces} \] commutes with geometric realizations. The same holds when replacing $\E{V}$ with $\Eiso{V}$.
\end{lemma}
\begin{proof}
    This follows from combining \Cref{prop:monadicity-forgetful} and \Cref{prop:monadicity-OmegaV}.
\end{proof}
\begin{lemma} \label{lem:H-fixed-point-geom-real}
    The $H$-fixed point functors
    \[ (-)^H \colon \Alg{\E{V}}{\GSpaces} \to \Alg{\E{V^H}}{\Spaces} \]
    commutes with geometric realizations. The same is true when replacing $\E{V}$ with $\Eiso{V}$. 
\end{lemma}
\begin{proof}
    This follows immediately from \Cref{prop:monadicity-forgetful}, as geometric realizations in both categories can be computed after applying the forgetful functor to $G$-spaces or spaces, respectively. For $G$-spaces, the $H$-fixed point functor commutes with all colimits.
\end{proof}
\subsection{Group completion for \texorpdfstring{$\Eiso{V}$}{EiV}-algebras}
As a first step, we will describe the group completion functor for $\Eiso{V}$-algebras, this step will be the main input to compute the effect of group completion on an $\E{V}$-algebra on $H$-fixed points for $H$ a $V$-isotropy subgroup.
%Let us start with the following proposition which enables us to compute $H$-fixed points.
The following proposition is the main geometric input we are using:
\begin{theorem}[Hauschild] \label{rourke-sanderson}
    Let $X$ be a based $G$-space and $H$ a $V$-isotropy subgroup of $G$. 
    Then the natural map \[ \GrpCompl{\E{V^H}} \left(  \Free{\Eiso{V}} X \right)^H \longrightarrow \left( \fgt{\E{V}}{\Eiso{V}} \Omega^V \Sigma^V X \right)^H \] is an equivalence of $\E{V^H}$-algebras. 
\end{theorem}
\begin{proof}
    %Using the explicit formula for a the free $\E{V}$-algebra \Cref{prop:formula-Kan-extension}, one can see that the restriction of the free $\Eiso{V}$-algebra in $G$-spaces to a subgroup $H$ is the free $\Eiso{\res{G}{H}V}$-algebra in $H$-spaces.
    As all functors and adjunctions used in this statement are parameterized, we can assume, without loss of generality, that $H=G$ and $\Eiso{V}=\E{V}$ because $G$ has $V$-isotropy or, equivalently, $V$ contains a trivial summand. 

    We will first argue, why this formula holds for $X=Y_+$ a based space obtained by adding a disjoint base point.
    In this case, we might use \cite[Rem. 2.72]{ste25} to identify 
    \[ \left( \Free{\Eiso{V}} Y_+ \right)^G \cong \coprod_{A \in \FinN^G} (\E{V}(A) \times \map(A,Y) )_{h\mathrm{Aut}(A)} \] where \[\E{V}(A) = \mathrm{Emb}(\oplus_A V,V) \simeq \mathrm{Conf}_A(V) \,. \] is the space of ordered equivariant configurations $A \to V$. 
    Modeling $Y$ by a $G$-CW-complex and computing the product and the homotopy quotient, which can be computed as a quotient, as the action is free, in the category of topological spaces, we see that $\left( \Free{\Eiso{V}} Y_+\right)^G$ is modeled by the space of unordered equivariant configurations in $V$, labeled in $Y$. This space is homeomorphic to the $G$-fixed points of the configuration space with its induced action, as defined in \cite{rs00}. A proof of this homeomorphism for the non-labeled version can be found in \cite[Prop. 3.2.10]{bqv23} and the same arguments apply in the case with labels.
    The loop space $\Omega^V \Sigma^V Y_+$ can also be modeled by the space of maps of topological spaces from $S^V$ to $\Sigma^V Y_+$. Under those identifications the result appears in \cite{rs00}. It is originally due to Hauschild \cite{hau80}, even though the published version only discusses the case $X=S^0$.

    Finally, we can deduce the based version from the non-based version, as the bar construction for $\E{0}$ provides a presentation of any based space as a geometric realization of spaces of the form $Y_+$. Moreover, all functors in question commute with geometric realizations by \Cref{lem:H-fixed-point-geom-real} and \Cref{lem:monadicity-E-Omega}. 
\end{proof}
\begin{remark}
    We could have avoided the extra step of deducing the formula for based $G$-spaces from the one for based $G$-spaces, as computing the free objects on unbased $G$-spaces is enough to deduce the recognition principle \Cref{thm:recognition-principle}, which in turn implies the approximation theorem \Cref{thm:approximation-theorem} also for based $G$-spaces. However, we decided that it is more natural to provide a proof of the approximation theorem for based $G$-spaces directly. 
\end{remark}
We use the above result to deduce that group completion of $\Eiso{V}$-algebras is computed pointwise:
\begin{proposition} \label{prop:grpcompl-le-H}
    Let $H$ be a $V$-isotropy and let $A$ be an $\Eiso{V}$-algebra. Then the natural map 
    \[ \GrpCompl{\E{V^H}} A^H \longrightarrow \left( \GrpCompl{\Eiso{V}} A \right)^H \] is an equivalence of $\Eiso{V}$-algebras. 
\end{proposition}
\begin{proof}
    We do construct the group completion functor directly. The functor 
    \[ \Omega^V \colon (\Spaces_\ast^G)_{\ge V} \longrightarrow \Alg{\Eiso{V}}{\GSpaces} \] commutes with limits as those are computed pointwise by \Cref{prop:monadicity-forgetful} and with filtered colimits as it does with sifted colimits, which can be seen as in \Cref{lem:monadicity-E-Omega}. 
    (Here, we abbreviated the functor $\fgt{\E{V}}{\Eiso{V}} \circ \Omega^V$ by $\Omega^V$). By \cite[Cor. 5.5.2.9]{lur09}, it hence does admit a left adjoint \[ B^V \colon \Alg{\Eisoext{V}}{\GSpaces} \longrightarrow (\Spaces_\ast^G)_{\ge V} \,. \] 
    We will prove that the unit \[ A \longrightarrow \Omega^V B^V A \] is a group completion on all fixed points. 
    This hence is a natural transformation for which the target always is group-like and an equivalence if the source is group-like. 
    It does follow that $\Omega^V B^V$ does compute the group completion as well as that this group completion is computed pointwise.

    If $A = \Free{\Eiso{V}} X$ for a based $G$-space $X$, we find that $B^V \Free{\Eiso{V}} X \cong \Sigma^V X$, so that the claim is precisely the content of \Cref{rourke-sanderson}.

    We will now write a general $\Eiso{V}$-algebra $A$ as a colimit of free algebras to reduce to this case. As we already know that the forgetful functor is monadic, we can consider the Bar construction $\BarC(A) \colon \Delta^{\op} \to \Alg{\Eiso{V}}{\GSpaces}$ from \cite[Exa. 4.7.2.7]{lur17} which resolves $A$ by free $\Eiso{V}$-algebras, so that $A \cong \colim_{\Delta^{\op}} \BarC(A)$. 
   
    Consider the diagram of $\E{V^H}$-algebras in spaces:
    \[ \begin{tikzcd}
        A^H \ar[r] \ar[d] & \left( \Omega^V B^V A \right)^H \ar[d] \\
        \colim_{\Delta^{\op}} \left( \BarC(A)^H \right) \ar[r] & \colim_{\Delta^{\op}} \left( \left( \Omega^V B^V \BarC(A) \right)^H \right)
    \end{tikzcd} \] where the horizontal arrows are induced by the unit of the adjunction and the vertical once are assembly maps for the colimit. As the left adjoint $B^V$ commutes with all colimits and
    \Cref{lem:H-fixed-point-geom-real} and \Cref{lem:monadicity-E-Omega} say that $\Omega^V$ and $H$-fixed points commute with geometric realizations too, we learn that the vertical morphisms are equivalences. 
    
    Now recall that $\BarC(A)$ is a resolution of $A$ by free algebras, for which we already argued above using \Cref{rourke-sanderson} that the unit is a group completion on $H$-fixed points. 
    We therefore presented the $H$-fixed points of the unit as a colimit of pointwise group completions and conclude that it is a group completion, which finishes the proof.
\end{proof}
\subsection{Group completion for \texorpdfstring{$\Eisoext{V}$}{EieV}-algebras}
In order to describe group completions for $\Eisoext{V}$-algebra, we want to use that such an algebra really just consists of the data of an $\Eiso{V}$-algebra together with a refinement of the underlying $\Eiso{0}$-algebra to an $\E{0}$-algebra. We prove this in the special case of interest for us, for algebras in $G$-spaces. 
\begin{lemma} \label{lem:descr-AlgEisoext}
    Let $\Spaces_\ast^{G,i}$ denote the category of based $G$-spaces with isotropy concentrated in $V$-isotropy subgroups, i.e. presheaves on the subcategory of the orbit category spanned by the corresponding orbits.
    The commutative diagram
    \[ \begin{tikzcd}[sep=huge]     \Alg{\Eisoext{V}}{\GSpaces} \ar[r,"{\fgt{\Eisoext{V}}{\Eiso{V}}}"] \ar[d,swap,"{\fgtOne{\Eisoext{V}}}"] & \Alg{\Eiso{V}}{\GSpaces} \ar[d,"\fgtOne{\Eiso{V}}"]  \\
    \Spaces_\ast^{G} \ar[r,"\mathrm{fgt}"] & \Spaces_\ast^{G,i}
    \end{tikzcd} \] 
    is a pullback.
\end{lemma}
\begin{proof}
    %We want to apply the monadicity theorem. In order to do so, we need to understand limits and colimits in the pullback. 

    %We claim that taking $H$-fixed points for $H$ being non-$V$-isotropy, i.e. the composite functor 
    %\[ \Alg{\Eiso{V}}{\GSpaces} \times_{\Spaces_\ast^{G,i}} \Spaces_\ast^G \longrightarrow \Alg{\Eiso{V}}{\GSpaces} \longrightarrow \Spaces_\ast^G \xrightarrow{(-)^H} \Spaces_\ast \] admits a right adjoint.
    %Given a based space $X$, we can send it to the element in the source which is given by the trivial $\Eiso{V}$-algebra in the first coordinate and the value $\beta_H X$ at $X$ of the right adjoint $\beta_H \colon \Spaces_\ast \to \Spaces_\ast^G$ to $(-)^H \colon \Spaces_\ast^G \to \Spaces_\ast$ in the second coordinate, becoming an element in the pullback by observing that $\beta_H X$ has trivial $K$-fixed points for $K$ being $V$-isotropy (or more generally, whenever $H$ is not subconjugate to $K$). 
    %In particular, the $H$-fixed points of this object are just $(\beta_H X)^H$. We claim that the counit $(\beta_H X)^H \to X$ exhibits this object as the right adjoint object of $X$ under the forgetful functor. 
    %This follows from the explicit description of mapping spaces in a pullback and using that the trivial $\Alg{\Eiso{V}}{\GSpaces}$ algebra is terminal, we leave the details to the reader. 
    
    We want to apply the monadicity theorem. We start by providing a description of the left adjoint to
    \[ \mathrm{pr}_2 \colon \Alg{\Eiso{V}}{\Spaces} \times_{\Spaces_\ast^{G,i}} \Spaces_\ast^G \longrightarrow \Spaces_\ast^G \,. \]
    Let $\mathrm{Free^i} \colon \Spaces_\ast^{G,i} \longrightarrow \Spaces_\ast^G$ denote the fully faithful left adjoint to the forgetful functor.
    
    Given $X$ a based $G$-space, let $F X$ be defined to be the pushout
    \begin{equation} \label{square-left-adjoint} \begin{tikzcd}
         \mathrm{Free}^i \mathrm{fgt}^i X \ar[d] \ar[r] & X \ar[d] \\
         \mathrm{Free}^i \fgtOne{\Eiso{V}}\Free{\Eiso{V}} \mathrm{fgt}^i X \ar[r]  & FX
    \end{tikzcd} \end{equation}
    where the upper map is the counit and the lower map is $\mathrm{Free}^i$ applied to the unit. 
    Then, using that $\mathrm{fgt}^i$ is fully faithful and preserves pushouts, one verifies that the lower map induces an equivalence $\mathrm{fgt}^i FX \cong \fgtOne{\Eiso{V}} \Free{\Eiso{V}} \mathrm{fgt}^i X$, so that $(\Free{\Eiso{V}} \mathrm{fgt}^i X,FX)$ lifts to an object in $\Alg{\Eiso{V}}{\Spaces} \times_{\Spaces_\ast^{G,i}} \Spaces_\ast^G$.
    We claim that the right arrow in the above square exhibits this object as the left adjoint object to $X$ under $\mathrm{pr}_2$. 
    This follows from a computation of mapping spaces in pullback categories which we leave to the reader.

    Let us now apply the monadicity theorem to the following diagram:
    \[  \begin{tikzcd}
        \Alg{\Eisoext{V}}{\GSpaces} \ar[rr,"{\left( \fgt{\Eisoext{V}}{\Eiso{V}},\fgt{\Eisoext{V}}{\E{0}} \right)}"] \ar[dr,swap,"{\fgtOne{\Eisoext{V}}}"] && \Alg{\Eiso{V}}{\GSpaces} \times_{\Spaces_\ast^{G,i}} \Spaces_\ast^G \ar[dl,"{\mathrm{fgt} \circ \mathrm{pr_2}}"] \\
        & \Spaces^G
    \end{tikzcd} \]
    Both functors are monadic by \Cref{prop:monadicity-forgetful}, as the right hand side is the category of algebras over the pushout $\E{0} \leftarrow \Eiso{0} \rightarrow \Eiso{V}$.
    
    Now we need to show that the top map preserves free objects. Let \[ \mathrm{Free} \colon \Spaces^G_\ast \longrightarrow \Alg{\Eiso{V}}{\GSpaces} \times_{\Spaces_\ast^{G,i}} \Spaces_\ast^G  \] denote the adjoint of $\mathrm{pr}_2$ which we constructed above. It is enough to show that for any based $G$-space $X$, the natural map 
    \begin{equation}  \label{final-monadicity-claim}
        \mathrm{Free} (X) \longrightarrow \left( \fgt{\Eisoext{V}}{\Eiso{V}},\fgtOne{\Eisoext{V}} \right) \left( \Free{\Eisoext{V}} X \right) \end{equation} is an equivalence.
    This can be checked on $H$-fixed points for $H$ all subgroups of $G$.
    As all the functors in question come from parameterized adjunctions, we might moreover assume that $H=G$.
    If $G$ itself is $V$-isotropy, the statement becomes trivial as $\Eiso{V}=\Eisoext{V}$. 
    %In this case, by the explicit construction of $\mathrm{Free}$, it is enough to argue that \[ (\mathrm{pr}_1 \mathrm{Free} X )^H \cong (\Free{\Eiso{V}} \mathrm{fgt}^i X)^H \longrightarrow (\fgt{\Eisoext{V}}{\Eiso{V}} \Free{\Eisoext{V}} X)^H \] 
    %is an equivalence.
    %This follows by using the formula from \Cref{prop:formula-Kan-extension} to compute both sides.

    So, we are left with the case where $G$ is not $V$-isotropy, i.e. $V^G = \{0\}$. 
    Here, we use that the left arrow in \eqref{square-left-adjoint} becomes an equivalence after applying $G$-fixed points, as $\mathrm{Free}^{i}$ just inserts the initial object $\ast$ on all fixed points which are not $V$-isotropy. We conclude that $( \mathrm{Free} (X) )^G \simeq X^G$ and unwinding definitions, we are left to show that
    %We have already seen that \[ \Alg{\Eiso{V}}{\GSpaces} \times_{\Spaces_\ast^{G,i}} \Spaces_\ast^G \longrightarrow \Spaces_\ast^G \xrightarrow{(-)^H} \Spaces_\ast \] admits a right adjoint whose given by the right adjoint to $(-)^H \colon \Spaces_\ast^G \to \Spaces_\ast$ in the second coordinate. Passing to left adjoints and evaluating at the based $G$-space $X$, it follows that the natural map \[ X^H \longrightarrow \left(\mathrm{Free} X\right)^H  \] is an equivalence. 
    %Therefore, in order to show that \eqref{final-monadicity-claim} is an equivalence on $H$-fixed points, it is enough to see that the natural map 
    \[ X^G \longrightarrow \left( \Free{\Eisoext{V}} X\right)^G \] is an equivalence. 
    This follows from the formula for the free algebra from \Cref{prop:formula-Kan-extension}, using that in the category $\Diskisoext{V}$, the only $G$-disks which embed into $V$, are the disk itself and the empty disk. %We leave the details to the reader. 
\end{proof}

Now we are ready to prove that group completion for $\Eisoext{V}$-algebras coincides with group completion on $\Eiso{V}$-algebras on $H$-fixed points for $H$ a $V$-isotropy subgroup and just does not change anything on all other fixed points:
\begin{proposition} \label{prop:Eisoext-grp-compl-is-levelwise}
    Let $A$ be a $\Eisoext{V}$-algebra and $H$ a subgroup of $G$. 
    \begin{itemize}
        \item If $H$ is $V$-isotropy, then the natural map 
        \[ \GrpCompl{\E{V^H}} A^H  \longrightarrow \left( \GrpCompl{\Eisoext{V}} A \right)^H \]
        is an equivalence of $\E{V^H}$-algebras.
        \item If $H$ is not $V$-isotropy, then the unit on $H$-fixed points
        \[ A^H \longrightarrow \left( \GrpCompl{\Eisoext{V}}A \right)^H \] is an equivalence of based spaces.
    \end{itemize}
\end{proposition}
\begin{proof}
    We will construct the group completion explicitly and check that it has the desired properties.
    For this we use the equivalence 
    \[ \Alg{\Eisoext{V}}{\GSpaces} \cong \Alg{\Eiso{V}}{\GSpaces} \times_{\Spaces_\ast^{G,i}} \Spaces_\ast^G \] from \Cref{lem:descr-AlgEisoext} which restricts to an equivalence  
    \[ \Alggrp{\Eisoext{V}}{\GSpaces} \cong \Alggrp{\Eiso{V}}{\GSpaces} \times_{\Spaces_\ast^{G,i}} \Spaces_\ast^G \,. \]
    Given \[ (A,\varphi \colon \fgtOne{\Eiso{V}}A \cong \mathrm{fgt}^i X,X) \in \Alg{\Eiso{V}}{\GSpaces} \times_{\Spaces_\ast^{G,i}} \Spaces_\ast^G \,, \] we construct its group completion as follows: 
    Its underlying $\Eiso{V}$-algebra is given by $\GrpCompl{\Eiso{V}} A$. Its underlying $G$-space is defined to be the pushout $P$ of
    \[ \begin{tikzcd}
        \mathrm{Free}^i(\fgtOne{\Eiso{V}} A) \ar[d] \ar[r] & X \ar[d] \\
        \mathrm{Free}^i (\fgtOne{\Eiso{V}} \GrpCompl{\Eiso{V}} A) \ar[r] & P
    \end{tikzcd} \] Here $\mathrm{Free}^i \colon \Spaces_\ast^{G,i} \to \Spaces_\ast^G$ denotes the left adjoint to the forgetful functor. The left map is the unit of the group completion adjunction. The upper map is the counit of the forgetful free adjunction, using the equivalence $\varphi \colon \fgtOne{\Eiso{V}}A \cong \mathrm{fgt}^i X$. 

    Moreover, the upper map becomes an equivalence after applying $\mathrm{fgt}^i \colon \Spaces^{G}_\ast \to \Spaces_\ast^{G,i}$, so that the lower map gives an equivalence \[ \bar \varphi \colon \fgtOne{\Eiso{V}} \GrpCompl{\Eiso{V}} A \cong \mathrm{fgt}^i \mathrm{Free}^i \fgtOne{\Eiso{V}} \GrpCompl{\Eiso{V}} A \cong \mathrm{fgt}^i P \,, \] so that we obtain an element \[ (\GrpCompl{\Eiso{V}}A,\bar \varphi,P) \in \Alg{\Eiso{V}}{\GSpaces} \times_{\Spaces_\ast^{G,i}} \Spaces_\ast^G \,. \]

    Now we have maps $X \to P$ and $A \to \GrpCompl{\Eiso{V}} A$ and by construction, a homotopy witnessing compatibility of the equivalences $\varphi$ and $\bar \varphi$ after applying $\mathrm{fgt}^i$ or $\fgtOne{\Eiso{V}}$, respectively, i.e. we constructed a map \[ (A,\varphi,X) \longrightarrow (\GrpCompl{\Eiso{V}}A,\bar \varphi,P) \,. \] Moreover, all this constructions can be made functorially, so that we constructed a natural transformation from any element in $\Alg{\Eiso{V}}{\GSpaces} \times_{\Spaces_\ast^{G,i}} \Spaces_\ast^G$ into a group-like element which is an equivalence if the source was group-like. It follows that this natural transformation exhibits the target as the group completion of the source. It moreover follows that this group completion is an equivalence on fixed points which are not $V$-isotropy and is given by group completion of the underlying $\Eiso{V}$-algebra on the fixed points which are $V$-isotropy. 
    Finally, it follows that this $\Eisoext{V}$-group completion is given by (non-equivariant) group completion on those fixed points by \Cref{prop:grpcompl-le-H}.
\end{proof}
\subsection{The free \texorpdfstring{$\E{V}$}{EV}-algebra on an \texorpdfstring{$\Eisoext{V}$}{EieV}-algebra}
In this section, we will compute the free $\E{V}$-algebra on an $\Eisoext{V}$-algebra $A$. 
The main observation is that the $H$-fixed points $A^H$ for $H$ not $V$-isotropy are acted on by the $H$-fixed points of the equivariant factorization homology $\int_{V \setminus \{0\}} A$ and that this is somehow \enquote{the only additional structure}, as the $H$-fixed points of the free $\E{V}$-algebra on $A$ will be the free $\left(\int_{V \setminus \{0\}} A\right)^H$-space on $A^H$.

\begin{construction} \label{constr:action-of-FH} 
    As already explained in \cite[Cor. 2.2]{lev22}, for $A$ an $\E{V}$-algebra and $H$ a subgroup which is not $V$-isotropy, there is a natural $\left(\int_{V \setminus \{0\}} A \right)$-action on $A$. It can be constructed as follows: The natural $\E{1}$-structure on $\R$ in $\mathrm{Man}$ (one-dimensional manifolds, in that case), gives rise to an $\E{1}$-structure on $V \setminus \{0\} \cong \R \times S(V)$ in $\Man{V}^G$ (even though this product is not a product as framed manifolds, the necessary embeddings are framed). Moreover, the $V$ is a module over $V \setminus \{0 \}$. Applying factorization homology then yields the action.
\end{construction}
\begin{proposition} \label{prop:free-EV-on-EVres}
    Let $A$ be an $\Eisoext{V}$-algebra in $G$-spaces and let $H$ be a subgroup of $G$.
    \begin{itemize}
        \item If $H$ is $V$-isotropy, then the induced map on $H$-fixed points of the unit \[ A^H \longrightarrow \left( \FreeTwo{\E{V}}{\Eisoext{V}} A \right)^H \] is an equivalence of $\E{V^H}$-algebras. 
        \item If $H$ is not $V$-isotropy, then by \Cref{constr:action-of-FH}, $\left(\FreeTwo{\E{V}}{\Eisoext{V}} A\right)^H$ has a natural $\left( \int_{\res{G}{H} V \setminus \{0\}} A \right)^H$-action, so that the unit uniquely extends to a $\left( \int_{\res{G}{H} V \setminus \{0\}} A \right)^H$-equivariant map \[ \left( \int_{\res{G}{H} V \setminus \{0\}} A \right)^H \times A^{H} \longrightarrow \left( \FreeTwo{\E{V}}{\Eisoext{V}} A \right)^{H} \,. \] out of the free $\left( \int_{\res{G}{H} V \setminus \{0\}} A \right)^H$-space on $A^H$. This map is an equivalence.
    \end{itemize}
    
\end{proposition}
\begin{proof}
    As all the adjunctions in question are parameterized, we might assume that $H=G$. The first part then becomes trivial as $\Eiso{V} = \Eisoext{V}$ in case that $G$ is $V$-isotropy. 

    We might therefore assume that $G$ is not $V$-isotropy.    
    Consider the functor \[ -\sqcup (V \xrightarrow{\id} V) \colon \left(\Diskiso{V} \times_{\Disk{V}} \overslice{\Disk{V}}{V}\right)^G  \longrightarrow \left( \Diskisoext{V} \times_{\Disk{V}} \overslice{\Disk{V}}{V} \right)^G \] adding one copy of $V$ with the identity structure map to $V$. We claim that the functor is cofinal. 
    Indeed, using that in $(\Diskisoext{V})^G$ only the empty disk and $V$ map into $V$, one sees that the functor is fully faithful. Using the same observation, one can also check that every object in the target has an initial morphism to one from the source (adding a disk if necessary), so that the functor is even a right adjoint inclusion.

    Now we apply \Cref{prop:formula-Kan-extension} to find: {\scriptsize
    \begin{align*} 
        &\left( \FreeTwo{\E{V}}{\Eisoext{V}} A \right) (V)  ^G \\
        \cong& \colim\left( \Diskisoext{V} \times_{\Disk{V}} \overslice{\Disk{V}}{V} \to \Diskisoext{V} \xrightarrow{A} \GSpaces \right)^G  \\
        \cong& \colim\left( \left(\Diskisoext{V} \times_{\Disk{V}} \overslice{\Disk{V}}{V}\right)^G \to \left(\Diskisoext{V}\right)^G \xrightarrow{A^G} \Spaces^G \xrightarrow{(-)^G} \Spaces \right) \\
        \cong& \colim\left( \left( \Diskiso{V} \times_{\Disk{V}} \overslice{\Disk{V}}{V} \right)^G \xrightarrow{\sqcup V} \left( \Diskisoext{V} \times_{\Disk{V}} \Disk{V} \right)^G \to (\Diskisoext{V})^G  \xrightarrow{A^G} \Spaces^G \xrightarrow{(-)^G} \Spaces \right) \\
        \cong& \colim\left( \left(\Diskiso{V} \times_{\Disk{V}} \overslice{\Disk{V}}{V} \right)^G \to \left(\Diskiso{V} \right)^G \xrightarrow{A^G} \Spaces^G \xrightarrow{(-)^G} \Spaces \xrightarrow{\times A(V)^G} \Spaces \right) \\
        \cong & \colim\left( \left( \Diskiso{V} \times_{\Disk{V}} \overslice{\Disk{V}}{V} \right)^G \to \left( \Diskiso{V}\right)^G \xrightarrow{A^G} \GSpaces \xrightarrow{(-)^G}\right) \times A(V)^G \\
        \cong& \left( \int_{V \setminus \{0\}} A\right)^G  \times A(V)^G
    \end{align*} }
    Here, we also used the fact \cite[Prop. 5.5]{sha23} that the $G$-fixed points of a $G$-colimit of a functor $\CC \to \GSpaces$ is naturally equivalent to the colimit of $\CC^G \to \GSpaces^G \to \Spaces$ where the last functor is taking $G$-fixed points.

   For the last isomorphism, we also used that the isotropy groups of $V \setminus \{0\}$ are $V$-isotropy and, moreover, any embedding of a $V$-isotropy disks into $V$ automatically lands in $V \setminus \{0\}$. It follows that the inclusion $V \setminus \{0\} \hookrightarrow V$ induces an equivalence \[ \left( \overslice{\Disk{V}}{V \setminus \{0\}} \right)^G \simeq  \left( \Diskiso{V} \times_{\Disk{V}} \overslice{\Disk{V}}{V} \right)^G \,. \] The colimit over the left hand side is, again by \cite[Prop. 5.5]{sha23}, the fixed points of the equivariant factorization homology. The inclusion $V \setminus \{0\} \to V$ is also the one used to define the action of the factorization homology on the fixed points in \Cref{constr:action-of-FH}, so the isomorphism is the one induced by the action.
\end{proof}
\subsection{Assembling the argument}
We are now ready to prove the approximation theorem by computing the free group-like $\E{V}$-algebra on a based space $X$. As explained in the proof strategy, we will first compute the free $\Eisoext{V}$-algebra on $X$, then determine its group completion using \Cref{prop:Eisoext-grp-compl-is-levelwise} and the result of Hauschild from \Cref{rourke-sanderson}. Then, we will compute the free $\E{V}$-algebra on this free group-like $\Eisoext{V}$-algebra using \Cref{prop:free-EV-on-EVres} and equivariant nonabelian Poincaré duality, \Cref{thm:NAPD}. 
\begin{theorem}[Approximation theorem] \label{theorem:approximation-theorem}
    For $X$ a based $G$-space, the natural map
    \[ \Free{\E{V}} X \longrightarrow \Omega^V \Sigma^V X \] from the free $\E{V}$-algebra on $X$ to the $V$-fold loop space of the $V$-fold suspension of $X$ exhibits the source as the group completion of the target. 
\end{theorem}
\begin{proof}
Consider the following commutative diagram of forgetful functors and inclusions of subcategories:
\[ \begin{tikzcd}
    \Alg{\E{V}}{\GSpaces}  \ar[r] &  \Alg{\Eisoext{V}}{\GSpaces} \ar[r] & \Alg{\E{0}}{\GSpaces} \\
    \Alggrp{\E{V}}{\GSpaces} \ar[r] \ar[u,hook] & \Alggrp{\Eisoext{V}}{\GSpaces} \ar[u,hook]
\end{tikzcd}
\]
It follows from the description in \Cref{prop:free-EV-on-EVres} that the free $\E{V}$-algebra on a group-like $\Eisoext{V}$-algebra is automatically group-like. We conclude that the adjunction between forgetful and free functor of $\E{V}$ and $\Eisoext{V}$ restricts to an adjunction on the subcategories of the respective group-like objects. Therefore, passing to left adjoints in the above diagram yields a natural equivalence 
\[ \GrpCompl{\E{V}} \Free{\E{V}} X \cong \FreeTwo{\E{V}}{\Eisoext{V}} \GrpCompl{\Eisoext{V}} \Free{\Eisoext{V}} X \,. \] %Combining this with our natural identification of the left hand side with $\Omega^V \Sigma^V X$ yields the result.

We can hence rephrase the theorem, asking whether the natural map 
\begin{equation} \label{comparison} 
\FreeTwo{\E{V}}{\Eisoext{V}} \GrpCompl{\Eisoext{V}} \Free{\Eisoext{V}} X \longrightarrow \Omega^V \Sigma^V X\end{equation} is an equivalence which is true if it is an equivalence on all fixed points. 
 
We start by showing that it is an equivalence on $H$-fixed points for $H$ being $V$-isotropy.
Recall from \Cref{rourke-sanderson} that the natural map
\[ \Free{\Eiso{V}} X \longrightarrow \fgt{\E{V}}{\Eiso{V}} \Omega^V \Sigma^V X \] is a group completion on $H$-fixed points for $H$ being $V$-isotropy. 
As $(\Diskiso{V})^H \hookrightarrow (\Diskisoext{V})^H$ is an equivalence for all $V$-isotropy $H$, we know that $\FreeTwo{\Eisoext{V}}{\Eiso{V}}(-)$ does not change those $H$-fixed points, so that \[ \Free{\Eisoext{V}} X \longrightarrow \fgt{\E{V}}{\Eisoext{V}} \Omega^V \Sigma^V X \] is a group completion on those $H$-fixed points, too. 
Now we can deduce from \Cref{prop:Eisoext-grp-compl-is-levelwise} that the natural map \[ \GrpCompl{\Eisoext{V}} \Free{\Eisoext{V}} X \longrightarrow \fgt{\E{V}}{\Eisoext{V}} \Omega^V \Sigma^V X \] is an equivalence on those $H$-fixed points, too. From \Cref{prop:free-EV-on-EVres} we finally deduce that the same holds for the map \eqref{comparison}.

%Note that we in particular already proved that the natural map \[ \fgt{\E{V}}{\Eiso{V}} \GrpCompl{\E{V}} \Free{\E{V}} X \longrightarrow \fgt{\E{V}}{\Eiso{V}} \Omega^V \Sigma^V X \] is an equivalence, which we will use in the next step.

Now we will turn towards the $H$-fixed points where $H$ is not $V$-isotropy.
In this step we will combine equivariant nonabelian Poincaré duality with a splitting of $(\Omega^V \Sigma^V X)^G$ due to Hauschild \cite{hau77}. We will again restrict to the case $H=G$, the other cases following from this by using that the free functor is a parameterized functor and hence restricts to a free functor on subgroups.

By \Cref{prop:free-EV-on-EVres}, we must verify that 
\[ X^G \to (\Omega^V \Sigma^V X)^G \] exhibits the target as the free space with an action of
\[ \left( \int_{V \setminus \{0\}} \Omega^V \Sigma^V X \right)^G \cong \Omega \map_\ast(S(V)_+,\Sigma^V X)^G \,, \] where the latter equivalence is due to \Cref{thm:NAPD}.
Let us unwind this acting on $(\Omega^V \Sigma^V X)^G$.
As nonabelian Poincaré duality is natural, this action is equivalently described as follows: The action of $V \setminus \{0\}$ on $V$ in $\Man{V}^G$ gives rise to an action of $V \setminus \{0\} ^+ \cong \Sigma S(V)_+$ on $V^+ = S^V$ in $(\Spaces_\ast^G)^\vee$. In other words, $S^V$ is a comodule over the coalgebra $\Sigma S(V)_+$ in based spaces. Finally, mapping out of those spaces yields the action of the algebra \[ \map_\ast(\Sigma S(V)_+,\Sigma^V X)^G \cong \Omega \map_\ast(\Sigma S(V)_+,\Sigma^V X)^G \] on $(\Omega^V \Sigma^V X)^G$.

This means that we need to check that the composite 
\begin{align*} 
    X^G \times \Omega \map_\ast(S(V)_+,\Sigma^V X)^G \longrightarrow& (\Omega^V \Sigma^V X)^G \times \Omega\map_\ast(S(V)_+,\Sigma^V X)^G\\
    \cong& \map_\ast(S^V \vee \Sigma S(V)_+,\Sigma^V X)^G \\
    \longrightarrow& (\Omega^V \Sigma^V X)^G 
    \end{align*} is an equivalence, where the first map is induced by the unit, and the second by the collapse map $S^V \to S^V \vee \Sigma S(V)_+$. Modeling $X$ by a $G$-CW complex, we can model the mapping spaces by taking mapping spaces in the category of based topological $G$-spaces. Using that $\Sigma S(V)_+ = S^V /S^0$, we can furthermore model 
    \[ \map(\Sigma S(V)_+,\Sigma^VX) \] by those maps $S^V \to \Sigma^VX$ which are sent to the base point in a neighborhood of the north and south pol.
    In this language, this splitting is due to Hauschild \cite{hau77}, proven in this precise form in \cite[Thm. 4]{rs00}. They only state that there is a splitting $(\Omega^V \Sigma^V X)^G \cong X^G \times Z$ for some group $Z$. However, in the proof \cite[pp. 11]{rs00}, they do check that $Z = \map_\ast(S(V)_+,\Sigma^V X)$ and that the homotopy equivalence $v \colon Z \times X^G \to (\Omega^V \Sigma^V X)^G$ is the one described above.  
This finishes the proof.
\end{proof}
\section{The recognition principle}
In this section, we will deduce the recognition principle from the approximation theorem, using the monadicity theorem.
\begin{proposition} \label{prop:grpcore}
    The inclusion 
    \[ \Alggrp{\E{V}}{\GSpaces} \longrightarrow \Alg{\E{V}}{\GSpaces} \]
    admits a right adjoint. 
\end{proposition}
\begin{proof}
    We first show that the inclusion 
    \[ \Alggrp{\E{V}}{\Set} \longrightarrow \Alg{\E{V}}{\Set} \] 
    admits a right adjoint where $\Set \subset \GSpaces$ is the $G$-symmetric monoidal subcategory of $G$-spaces which is given by spaces which are discrete on all fixed points. 

    Given $A \in \Alg{\E{V}}{\Set}$, let $\grpcore{A}$ denote the $\E{V}$-algebra obtained from $A$ by passing on $H$-fixed points $A^H$ to the subset of elements for which the image in $A^K$ for any subconjugate $K$ of $H$ which is $V$-isotropy is invertible in $A^K$. The $\E{V}$-structure of $A$ does restrict to one on $\grpcore{A}$. Moreover, the inclusion $\grpcore{A} \to A$ gives a natural transformation which is the identity if $A$ is group-like and the source always is group-like. It hence exhibits $\grpcore{(-)}$ as the right adjoint to the inclusion.

    Now for $A \in \Alg{\E{V}}{\GSpaces}$, let $\grpcore{A}$ denote the pullback 
    \[ \begin{tikzcd}
        \grpcore{A} \ar[d] \ar[r] & A \ar[d] \\
        \grpcore{\pi_0(A)} \ar[r] & \pi_0(A)
    \end{tikzcd}\]
    where the vertical morphisms are induced by the unit $A \to \pi_0(A)$ of the left adjoint $\pi_0 \colon \GSpaces \to \Set$ (which is $G$-symmetric monoidal). 

    Using that pullbacks can be computed pointwise by \Cref{prop:monadicity-forgetful}, one verifies that $\grpcore{A} \to A$ again is a natural transformation which is an equivalence on group-like objects and the source always is group-like, exhibiting $\grpcore{(-)}$ as the right adjoint to the inclusion.
\end{proof}
\begin{proof}[Proof of \Cref{thm:recognition-principle}]
    We actually start by showing that $\Omega^V \colon (\Spaces^G_\ast)_{\ge V} \to \Alggrp{\E{V}}{\GSpaces}$ is an equivalence. 
    We will use the monadicity theorem \cite[Cor. 4.7.3.16]{lur17} applied to the following situation 
    \[ \begin{tikzcd}
        (\Spaces^G_\ast)_{\ge V} \ar[rr,"\Omega^V"] \ar[dr,swap,"\Omega^V"] & & \Alggrp{\E{V}}{\GSpaces} \ar[dl,"\fgtOne{\E{V}}"] \\
        & \Spaces^G
    \end{tikzcd} \,. \]
    We need to verify that both functors to $\Spaces^G$ are monadic and that the third functor preserves free objects. 

    The left hand map is monadic by \Cref{prop:monadicity-OmegaV}. The forgetful functor $\Alg{\E{V}}{\GSpaces} \to \Spaces^G$ is monadic by \Cref{prop:monadicity-forgetful}. Since the inclusion $\Alggrp{\E{V}}{\GSpaces} \to \Alg{\E{V}}{\GSpaces}$ does commute with all colimits by \Cref{prop:grpcore} and also is conservative right adjoint by \Cref{prop:grpcompl-exists}, the right arrow in the above diagram is monadic, too.

    Finally, the comparison map between free objects is an equivalence by the approximation theorem, using that the two functors to $\Spaces^G$ factor through the forgetful functor $\Spaces^G_\ast \to \Spaces^G$ which admits a left adjoint, i.e. we apply \Cref{theorem:approximation-theorem} to the case where $X$ arises from a non-based $G$-space by adding a disjoint base point. It follows that \[ \Omega^V \colon (\Spaces^G_\ast)_{\ge V} \longrightarrow \Alggrp{\E{V}}{\GSpaces} \] is an equivalence, as desired.

    As limits are computed pointwise, the inclusion $(\Spaces_\ast^G)_{\ge V} \subset \Spaces_\ast^G$ admits a right adjoint, the $V$-connective cover. As $S^V$ is $V$-connective, the functor $\Omega^V \colon \Spaces^G_\ast \to \Alg{\E{V}}{\GSpaces}$ factors through this $V$-connective cover functor.
    
    It follows from composing left adjoints that $\Omega^V \colon \Spaces^G_\ast \longrightarrow \Alg{\E{V}}{\GSpaces}$ admits a left adjoint $B^V \colon \Alg{\E{V}}{\GSpaces} \to \Spaces_\ast^G$ with the desired properties stated in \Cref{thm:recognition-principle} which is given by first group completing, then using the equivalence between group-like $\E{V}$-algebras and $V$-connective based $G$-spaces established above and then finally regarding the resulting $V$-connective based $G$-spaces as just a based $G$-spaces. We will however describe this adjoint more explicitly in the upcoming \Cref{rem:description-BV}
\end{proof}
\begin{remark} \label{rem:description-BV}
    Given an $\E{V}$-algebra $A$, recall its bar complex $\BarC(A) \colon \Delta^{\op} \to \Alg{\E{V}}{\GSpaces}$ given by resolving $A$ by free algebras, that is \[ \BarC(A)_n = \left( \Free{\E{V}} \right)^n (\fgtOne{\E{V}} A) \] as $B^V$ is a left adjoint we have 
    \[ B^V A \cong B^V \colim_{\Delta^{\op}} \BarC(A) \cong \colim_{\Delta^{\op}} B^V \BarC(A) \] where
    \[ \left(B^V \BarC(A)\right)_n = \left(B^V \circ \left(\Free{\E{V}}\right)^n \right) \left(\fgtOne{\E{V}}A \right) \cong \Sigma^V \left(\Free{\E{V}}\right)^{n-1} (\fgtOne{\E{V}} A) \,, \] that is we can compute the delooping $B^V$ using a two-sided bar construction.
\end{remark}
\printbibliography

@misc{bqv23,
      title={Bredon homological stability for configuration spaces of $G$-manifolds}, 
      author={Eva Belmont and J. D. Quigley and Chase Vogeli},
      year={2023},
      eprint={2311.02459},
      archivePrefix={arXiv},
      primaryClass={math.AT},
}

@online{bhs24,
    author={Barkan, Shaul and Haugseng, Rune and Steinebrunner, Jan},
    title={Envelopes for algebraic patterns},
    eprint={2208.07183},
    eprinttype={arxiv},
    eprintclass={math.CT},
    year={2024}
}

@online{bdgns24,
    author={Barwick, Clarck and Dotto, Emanuele and Glasman, Saul and Nardin, Denis and Shah, Jay},
    title={Parametrized higher category theory and higher algebra: Exposé I -- Elements of parametrized higher category theory},
    eprint={1608.03657},
    eprinttype={arxiv},
    eprintclass={math.AT},
    year={2016}
}

@book {bv73,
    AUTHOR = {Boardman, J. Michael and Vogt, Rainer M.},
     TITLE = {Homotopy invariant algebraic structures on topological spaces},
    SERIES = {Lecture Notes in Mathematics},
    VOLUME = {Vol. 347},
 PUBLISHER = {Springer-Verlag, Berlin-New York},
      YEAR = {1973},
     PAGES = {x+257},
}

@article {cw91,
    AUTHOR = {Costenoble, Steven R. and Waner, Stefan},
     TITLE = {Fixed set systems of equivariant infinite loop spaces},
   JOURNAL = {Trans. Amer. Math. Soc.},
  FJOURNAL = {Transactions of the American Mathematical Society},
    VOLUME = {326},
      YEAR = {1991},
    NUMBER = {2},
     PAGES = {485--505},
}

@article {dmpr21,
    AUTHOR = {Dotto, Emanuele and Moi, Kristian and Patchkoria, Irakli and
              Reeh, Sune Precht},
     TITLE = {Real topological {H}ochschild homology},
   JOURNAL = {J. Eur. Math. Soc. (JEMS)},
  FJOURNAL = {Journal of the European Mathematical Society (JEMS)},
    VOLUME = {23},
      YEAR = {2021},
    NUMBER = {1},
     PAGES = {63--152},
}

@article{hau77,
  title={Zerspaltung {\"a}quivarianter Homotopiemengen},
  author={Hauschild, Henning},
  journal={Mathematische Annalen},
  volume={230},
  number={3},
  pages={279--292},
  year={1977}
}

@InProceedings{hau80,
author="Hauschild, Henning",
editor="Koschorke, Ulrich
and Neumann, Walter D.",
title="{\"A}quivariante Konfigurationsr{\"a}ume und Abbildungsr{\"a}ume",
booktitle="Topology Symposium Siegen 1979",
year="1980",
publisher="Springer Berlin Heidelberg",
pages="281--315",
}

@online{hor19,
	AUTHOR = {Horev, Asaf},
	TITLE = {Genuine equivariant factorization homology},
	eprint={1910.07226},
	eprinttype ={arxiv},
	eprintclass = {math.AT},
	year={2019}
}

@online{hkz24,
	AUTHOR = {Horev, Asaf and Klang, Inbar and Zou, Foling},
	TITLE = {Equivariant nonabelian Poincaré duality and equivariant factorization homology of Thom spectra},
	eprint={2006.13348},
	eprinttype ={arxiv},
	eprintclass = {math.AT},
	year={2024}
}

@article {gm17,
    AUTHOR = {Guillou, Bertrand J. and May, J. Peter},
     TITLE = {Equivariant iterated loop space theory and permutative
              {$G$}-categories},
   JOURNAL = {Algebr. Geom. Topol.},
  FJOURNAL = {Algebraic \& Geometric Topology},
    VOLUME = {17},
      YEAR = {2017},
    NUMBER = {6},
     PAGES = {3259--3339},
}

@article {hw20,
    AUTHOR = {Hahn, Jeremy and Wilson, Dylan},
     TITLE = {Eilenberg--{M}ac {L}ane spectra as equivariant {T}hom spectra},
   JOURNAL = {Geom. Topol.},
  FJOURNAL = {Geometry \& Topology},
    VOLUME = {24},
      YEAR = {2020},
    NUMBER = {6},
     PAGES = {2709--2748},
}

@online{llp25,
      title={Norms in equivariant homotopy theory}, 
      author={Tobias Lenz and Sil Linskens and Phil Pützstück},
      year={2025},
      eprint={2503.02839},
      eprinttype={arXiv},
      eprintclass={math.AT},
      %url={https://arxiv.org/abs/2503.02839}, 
}

@article {lev22,
    AUTHOR = {Levy, Ishan},
     TITLE = {Eilenberg {M}ac {L}ane spectra as {$p$}-cyclonic {T}hom
              spectra},
   JOURNAL = {J. Topol.},
  FJOURNAL = {Journal of Topology},
    VOLUME = {15},
      YEAR = {2022},
    NUMBER = {2},
     PAGES = {878--895},
}

@book {lur09,
    AUTHOR = {Lurie, Jacob},
     TITLE = {Higher topos theory},
    SERIES = {Annals of Mathematics Studies},
    VOLUME = {170},
 PUBLISHER = {Princeton University Press, Princeton, NJ},
      YEAR = {2009},
     PAGES = {xviii+925},
}

@online{lur17,
    author={Lurie, Jacob},
    title={Higher Algebra},
    url={https://people.math.harvard.edu/~lurie/papers/HA.pdf},
    year={2017}
}

@book {may72,
    AUTHOR = {May, J. Peter},
     TITLE = {The geometry of iterated loop spaces},
    SERIES = {Lecture Notes in Mathematics},
    VOLUME = {Vol. 271},
 PUBLISHER = {Springer-Verlag, Berlin-New York},
      YEAR = {1972},
     PAGES = {viii+175},
   MRCLASS = {55D35},
  MRNUMBER = {420610},
MRREVIEWER = {J.\ Stasheff},
}

@book {clm76,
    AUTHOR = {Cohen, Frederick R. and Lada, Thomas J. and May, J. Peter},
     TITLE = {The homology of iterated loop spaces},
    SERIES = {Lecture Notes in Mathematics},
    VOLUME = {Vol. 533},
 PUBLISHER = {Springer-Verlag, Berlin-New York},
      YEAR = {1976},
     PAGES = {vii+490},
   MRCLASS = {55G25 (55D35)},
  MRNUMBER = {436146},
MRREVIEWER = {Peter\ J.\ Eccles},
}

@article {moi20,
    AUTHOR = {Moi, Kristian Jonsson},
     TITLE = {Equivariant loops on classifying spaces},
   JOURNAL = {Algebr. Geom. Topol.},
  FJOURNAL = {Algebraic \& Geometric Topology},
    VOLUME = {20},
      YEAR = {2020},
    NUMBER = {5},
     PAGES = {2511--2552},
}

@online{ns22,
	AUTHOR = {Nardin, Denis and Shah, Jay},
	TITLE = {Parametrized and equivariant higher algebra},
	eprint={2203.00072},
	eprinttype ={arxiv},
	eprintclass = {math.AT},
	year={2022}
}

@article {rs00,
    AUTHOR = {Rourke, Colin and Sanderson, Brian},
     TITLE = {Equivariant configuration spaces},
   JOURNAL = {J. London Math. Soc. (2)},
  FJOURNAL = {Journal of the London Mathematical Society. Second Series},
    VOLUME = {62},
      YEAR = {2000},
    NUMBER = {2},
     PAGES = {544--552},
}

@article {seg73,
    AUTHOR = {Segal, Graeme},
     TITLE = {Configuration-spaces and iterated loop-spaces},
   JOURNAL = {Invent. Math.},
  FJOURNAL = {Inventiones Mathematicae},
    VOLUME = {21},
      YEAR = {1973},
     PAGES = {213--221},
}

@article {sha23,
    AUTHOR = {Shah, Jay},
     TITLE = {Parametrized higher category theory},
   JOURNAL = {Algebr. Geom. Topol.},
  FJOURNAL = {Algebraic \& Geometric Topology},
    VOLUME = {23},
      YEAR = {2023},
    NUMBER = {2},
     PAGES = {509--644},
}

@online{ste25,
    AUTHOR={Stewart, Natalie},
    TITLE={Equivariant operads, symmetric sequences, and Boardman-Vogt tensor products},
    eprint={2501.02129},
    eprinttype={arxiv},
    eprintclass={math.CT},
    year={2025}
}

@book {sti13,
    AUTHOR = {Stiennon, Nisan},
     TITLE = {The {M}oduli {S}pace of {R}eal {C}urves and a
              {Z}/2-{E}quivariant {M}adsen-{W}eiss {T}heorem},
      NOTE = {Thesis (Ph.D.)--Stanford University},
 PUBLISHER = {ProQuest LLC, Ann Arbor, MI},
      YEAR = {2013},
     PAGES = {74},
}

@article {zou23,
    AUTHOR = {Zou, Foling},
     TITLE = {A geometric approach to equivariant factorization homology and
              nonabelian {P}oincar\'e{} duality},
   JOURNAL = {Math. Z.},
  FJOURNAL = {Mathematische Zeitschrift},
    VOLUME = {303},
      YEAR = {2023},
    NUMBER = {4},
     PAGES = {Paper No. 98, 47},
}
\end{document}